\documentclass[12pt]{amsart}
\usepackage{amssymb}
\input amssym.def
\usepackage{amsmath,amsfonts,hyperref,xcolor,mathtools}
\usepackage{amscd}
\usepackage[mathscr]{eucal}
\usepackage{enumitem}
\usepackage{orcidlink}

\hypersetup{colorlinks=true,citecolor={purple},linkcolor={teal},urlcolor={violet}}

\newcommand{\N}{{\mathbb N}}
\newcommand{\Z}{{\mathbb Z}}

\newcommand{\C}{{\mathbb C}}
\newcommand{\R}{{\mathbb R}}

\newtheorem*{mthm}{Main Theorem}
\newtheorem{thm}{Theorem}
\newtheorem{lem}{Lemma}
\newtheorem{cor}{Corollary}
\newtheorem{prop}{Proposition}

\newtheorem{rmk}{Remark}
\newtheorem{defn}{Definition}
\newtheorem{ex}{Example}

\newcommand{\thmref}[1]{Theorem~\ref{#1}}
\newcommand{\propref}[1]{Proposition~\ref{#1}}
\newcommand{\lemref}[1]{Lemma~\ref{#1}}
\newcommand{\corref}[1]{Corollary~\ref{#1}}

\makeatletter
\@namedef{subjclassname@2020}{%
\textup{2020} Mathematics Subject Classification}
\makeatother

\begin{document}


\title[Multiple polylogarithms, regularisation and open domain of convergence]
{Multiple polylogarithms, a regularisation process and an admissible open domain of convergence}

\author{Pawan Singh Mehta \orcidlink{0009-0001-6717-9063} and Biswajyoti Saha\orcidlink{0009-0009-2904-4860}}

\address{Pawan Singh Mehta and Biswajyoti Saha\\ \newline
Department of Mathematics, Indian Institute of Technology Delhi, 
Hauz Khas, New Delhi 110016, India.}
\email{maz218521@maths.iitd.ac.in, biswajyoti@maths.iitd.ac.in}

\subjclass[2020]{11M32}

\keywords{multiple polylogarithms, open domain of convergences, translation formula,
regularisation process, generalised Euler-Boole summation formula}

\begin{abstract}
In this article, we study the analytic properties of the multiple polylogarithms in the $s$-aspect.
Although the domain of absolute convergence of the series defining the multiple polylogarithms
 is well-known, the study towards a
larger open domain of (conditional) convergence has been limited, particularly when the depth is $\ge 2$.
Here, we exhibit a larger open domain of (conditional) convergence for this series by writing certain translation
formulas satisfied by them. The series moreover defines a holomorphic function in this open set.
We then introduce a regularisation process for the multiple polylogarithms, extending an earlier
work of the second author. This regularisation process requires a generalisation of the Euler-Boole summation formula
that we derive in the appendix of this article. The regularisation process leads to a larger open domain, where
the series (conditionally) converges at integer points. The holomorphicity at such points is a more delicate
question and this regularisation process is to be used to study the local behaviour of the
multiple polylogarithms around such points.
\end{abstract}

\maketitle

\section{Introduction}

Throughout this article, we denote the set of all non-negative integers by $\N$.
Let $a$ be a positive integer and $z$ be a complex number with $\lvert z\rvert<1$. The classical
polylogarithm function is defined by the convergent power series 
$$
\mathrm{Li}(a;z):=\sum_{n>0}\frac{z^n}{n^a}.
$$
Convergence of such a series depends on $z$ and $a$ both. For example, we can take $\lvert z\rvert\leq1$,
if $a\geq 2$. However, for $\lvert z\rvert>1$, the series is not convergent.

For a fixed value of $z$ with $\lvert z\rvert\leq 1$, the polylogarithm functions can be considered as the
Dirichlet series with the integer $a$ replaced by a complex variable $s$ such that $\Re(s)>1$, i.e.
$$
\mathrm{Li}_z(s):=\sum_{n>0}\frac{z^n}{n^s}.
$$
Note that the above Dirichlet series converges (absolutely) for every complex number $s$ if $|z|<1$.
Now if $|z|=1$, the domain of convergence of this Dirichlet series depends on the value of $z$.
For $z=1$, we get back the ubiquitous Riemann zeta function, which converges (absolutely) for $\Re(s)>1$.
However, for $|z|=1$ with $z\neq 1$, the Dirichlet series $\mathrm{Li}_z(s)$ converges for $\Re(s)>0$.
This is immediate as the partial sum $\sum_{N>n>0}z^n=\frac{z-z^N}{1-z}$ is bounded. This also shows
that in this case, the Dirichlet series $\mathrm{Li}_z(s)$ does not converge for any $s$ with $\Re(s)<0$.

The multiple polylogarithms are defined by considering several variable analogue of the classical polylogarithm functions.
For positive integers $r$, $a_1, \ldots, a_r$ and complex number $z_1,\ldots, z_r$ with $|z_i|<1$ for all $1\leq i\leq r$,
the multiple polylogarithm (of depth $r$) is defined by the convergent power series
 $$
 \mathrm{Li}(a_1,\ldots,a_r;z_1,\ldots,z_r):=\sum_{n_1>\cdots>n_r>0}\frac{z_1^{n_1}\cdots z_r^{n_r}}{n_1^{a_1}\cdots n_r^{a_r}}.
 $$
Again, this definition can be extended for $|z_i|\leq 1$ for all $1\leq i\leq r$, if $a_1\geq2$.
Further, this series can be considered as a multiple Dirichlet series by fixing complex numbers
$z_1,\ldots, z_r$ and replacing $a_1,\ldots,a_r$ by complex variables $s_1,\ldots,s_r$
with $\Re(s_1+\cdots+s_i)>i$ for $1 \le i \le r$, i.e.,
\begin{equation}\label{mp}
\mathrm{Li}_{(z_1,\ldots,z_r)}(s_1,\ldots,s_r):=\sum_{n_1>\cdots>n_r>0}\frac{z_1^{n_1}\cdots z_r^{n_r}}{n_1^{s_1}\cdots n_r^{s_r}},
\end{equation}
where $(s_1, \ldots, s_r) \in U_r$ with
$$
U_r := \{ (s_1, \ldots, s_r) \in \C^r ~:~ \Re(s_1 + \cdots + s_i) > i
~\text{ for all }~ 1 \le i  \le r \}.
$$
The above series converges normally on compact subsets of $U_r$ and 
therefore defines a holomorphic function there (see \cite[Proposition 2]{BS1}).
The second author further showed in \cite[Theorem 9]{BS1} (albeit using different notations)
that the multiple Dirichlet series \eqref{mp} can be extended to a meromorphic function on the
whole of $\C^r$ and the set of singularities depends on 
the values of the product $z_{[1,i]}:=z_1 \cdots z_i$ for $1 \le i \le r$.

\begin{thm}[Saha]\label{poles-malf}
The function $\mathrm{Li}_{(z_1,\ldots,z_r)}(s_1,\ldots,s_r)$
extends to a meromorphic function on the whole of $\C^r$.
If $z_{[1,i]} \neq 1$ for all $1 \le i \le r$, then
$\mathrm{Li}_{(z_1,\ldots,z_r)}(s_1,\ldots,s_r)$ is holomorphic on $\C^r$.
Otherwise, if $i_1 < \cdots < i_m$ denote all the indices such that $z_{[1,i_j]}= 1$
for all $1 \le j \le m $, then the set of all possible singularities of
$\mathrm{Li}_{(z_1,\ldots,z_r)}(s_1,\ldots,s_r)$ can be given as follows:

$(a)$ If $i_1=1$, then $\mathrm{Li}_{(z_1,\ldots,z_r)}(s_1,\ldots,s_r)$ is
holomorphic outside the union of the hyperplanes given by the equations
$$
s_1=1; ~s_1 + \cdots+ s_{i_j}=n \ \text{ for all } n \in \Z_{\le j} \ \text{and }
2 \le j \le m.
$$

$(b)$ If $i_1 \not =1$, then $\mathrm{Li}_{(z_1,\ldots,z_r)}(s_1,\ldots,s_r)$ is
holomorphic outside the union of the hyperplanes given by the equations
$$
s_1 + \cdots+ s_{i_j}=n \text{ for all } n \in \Z_{\le j} \text{ and }
1 \le j \le m.
$$
\end{thm}

This indicates that there could be an open domain of convergence bigger than $U_r$
for the series in \eqref{mp}, while treating it as the limit
\begin{equation}\label{lim-mp}
\lim_{N \to \infty}\sum_{N>n_1>\cdots>n_r>0}\frac{z_1^{n_1}\cdots z_r^{n_r}}{n_1^{s_1}\cdots n_r^{s_r}},
\end{equation}
depending on the values of $z_i$, exactly as in the case of depth $r=1$. 
This question does not seem to have been addressed before.
If all the $z_i$'s are of unit modulus, the convergence of the series in \eqref{mp} outside $U_r$ will be conditional
convergence and hence we follow the above convention for the convergence going forward.
It is easy to see that for $s \in \C$ with $\Re(s)>1$, the series 
$\sum_{n_1>n_2>0}\frac{(-1)^{n_2}}{n_1^{s}}$ converges and the sum is $-2^{-s}\zeta(s)$.
One can also see that the series $\sum_{n_1>n_2>0}\frac{(-1)^{n_1}}{n_1}$ does not converge,
although the multiple Dirichlet series $\mathrm{Li}_{(-1,1)}(s_1,s_2)$ is holomorphic everywhere in
$\C^2$. Hence there are some inherent intricacies. This brings us to our first theorem. We need some notations.

For $r\geq 1$ an integer and ${\bf z}:=(z_1,\ldots,z_r) \in \C^r$  where $|z_i|\leq1$ for all $1\leq i\leq r$,
let $q({\bf z})$ be the smallest positive integer such that $1\leq q({\bf z})\leq r$ and the product $z_{[1,q({\bf z})]}\neq1$.
If no such $q({\bf z})$ exists, then we set $q({\bf z})=r+1$. We define an open set $U_r(\mathbf{z})$ of $\mathbb{C}^r$ as follows:
\begin{equation}\label{Urz}
\begin{split}
U_r(\mathbf{z}):=\{ & (s_1,\ldots,s_r)\in\mathbb{C}^r: \Re(s_1+\cdots+s_i)>i \text{ if } 1\leq i < q({\bf z}) \\
&\text{ and } \Re(s_1+\cdots+s_i)>i-1 \text{ if } q({\bf z})\leq i \leq r \} .
\end{split}
\end{equation}
Note that in general, we have $U_r \subseteq U_r(\mathbf{z})$. Further,  $U_r(\mathbf{z})=U_r$
if and only if $z_i=1$ for all $1\leq i\leq r$ (i.e., $q({\bf z})=r+1$).
In this context, we prove the following theorem.

\begin{thm}\label{big-dom}
With the notations and convention as above, the multiple Dirichlet series in \eqref{mp}
converges on $U_r(\mathbf{z})$. Moreover, it defines a holomorphic function on $U_r(\mathbf{z})$.
\end{thm}

This theorem therefore extends the known result for the depth $1$ Dirichlet series $\mathrm{Li}_z(s)$.
Also, in the last example, note that the point $(1,0)$ lies on the boundary of $U_2(-1,1)$.
However, this open domain of convergence is not optimal in general and we can do better.
For example, as $N$ tends to $\infty$, one has
\begin{equation}\label{ex-0}
\sum_{N>n_1>n_2>0}\frac{(-1)^{n_2}}{n_1^{2}n_2^{-1}}=\frac{\log 2}{2}-\frac{\pi^2}{16}+o(1).
\end{equation}
Clearly, $(2,-1) \notin U_2(1,-1)$. In fact, as we can see, for $s \in \C$ with $\Re(s)>1$, the limit
$$
\lim_{N \to \infty} \sum_{N>n_1>n_2>0}\frac{(-1)^{n_2}}{n_1^{s}n_2^{-1}}
$$
exists. Our next aim in this article is therefore to establish a more accurate and larger open domain
where the series in \eqref{mp} (viewed as the limit in \eqref{lim-mp}) converges at such interesting integer points.
To do this we need to set up a regularisation process for the multiple polylogarithms (at roots of unity), following the work
of the second author \cite{BS2}. This is done in the next section. Using this regularisation process, for roots of unity $z_1,\ldots,z_r$,
we can give a more accurate and larger open domain where the multiple Dirichlet series $\mathrm{Li}_{(z_1,\ldots,z_r)}(s_1,\ldots,s_r)$
converges at integer points. The optimality of this admissible open domain of convergence can be explained
in light of the local behaviour of the multiple polylogarithms at integral points, which is going to appear in a forthcoming article \cite{MS}.

To state our theorem in this context, we will need a few more notations. Let $r\geq 1$ be an integer and $\mathbf{z}=(z_1,\ldots,z_r) \in \C^r$
be such that $z_1,\ldots,z_r$ are roots of unity. For $1 \le i \le j \le r$, consider the products $z_{[i,j]}:=z_i \cdots z_j$ and 
for each $1 \le j \le r$, define
\begin{equation}\label{I-Q}
I_j(\mathbf{z}):=\{i : 1 \le i \le j \text{ and } z_{[i,j]}=1\} \ \text{ and } \ Q_j(\mathbf{z}):=|I_j(\mathbf{z})|,
\end{equation}
the number of elements in $I_j(\mathbf{z})$.
Note that in general there is no comparing relations among various $Q_j(\mathbf{z})$ for $1 \le j \le r$.
Consider the open set
\begin{equation}\label{Vrz}
V_r(\mathbf{z})=\{(s_1,\ldots,s_r)\in \C^r : \Re(s_1+\cdots+s_i)>Q_i(\mathbf{z}) \text{ for all } 1\leq i\leq r\}.
\end{equation}
It is easy to see that $V_r(\mathbf{z})=U_r$ if and only if $\mathbf{z}=(1,\ldots,1)$. Moreover,
$U_r(\mathbf{z}) \subseteq V_r(\mathbf{z})$. This follows from the fact that
$Q_i(\mathbf{z}) \le i$ for all $1 \le i \le r$ and further $Q_i(\mathbf{z}) \le i-1$ for all $q(\mathbf{z}) \le i \le r$.
Note that $U_2(-1,1)=V_2(-1,1)$ and $U_2(-1,-1)=V_2(-1,-1)$, but $U_2(1,-1) \subsetneq V_2(1,-1)$,
as $(2,-1) \in V_2(1,-1) \setminus U_2(1,-1)$. Moreover, the special value of the series
ar $(2,-1)$ is $\frac{\log 2}{2}-\frac{\pi^2}{16}$ (see \eqref{ex-0}).
In this regard, we prove the following theorem.

\begin{thm}\label{ad-dom}
Let $r \ge 1$ and $z_1,\ldots,z_r \in \C$ be roots of unity. The multiple Dirichlet series
converges at integral points of $V_r(\mathbf{z})$. Moreover, for an integral point $(s_1,\ldots,s_r) \in V_r(\mathbf{z})$
and $k_1,\ldots, k_r\in \mathbb{N}$, the series
\begin{equation}\label{mp-der}
\sum_{n_1>\cdots>n_r>0}\frac{z_1^{n_1}(\log n_1)^{k_1}\cdots z_r^{n_r}(\log n_r)^{k_r}}{n_1^{s_1}\cdots n_r^{s_r}}
\end{equation}
converges when we consider it as the limit
\begin{equation}\label{lim-mp-der}
\lim_{N \to \infty}\sum_{N>n_1>\cdots>n_r>0}\frac{z_1^{n_1}(\log n_1)^{k_1}\cdots z_r^{n_r}(\log n_r)^{k_r}}{n_1^{s_1}\cdots n_r^{s_r}}.
\end{equation}
\end{thm}

Our proof of \thmref{big-dom} uses translation formulas satisfied by the partial tails of the
series in \eqref{mp}. A prototype of such translation formulas can be found in \cite[Theorems 6, 7]{BS1} (which was
derived for the whole sum). The proof of  \thmref{big-dom} is given in Section 3.

But this method does not seem to yield \thmref{ad-dom}. For a proof of \thmref{ad-dom}, we need to set up
a regularisation process for the multiple polylogarithms (at the roots of unity). The regularisation process allows us 
to assign a (regularised) value of the series in \eqref{mp-der}, even when it does not converge. But more importantly,
this regularised value would be equal to the sum of the series in \eqref{mp-der} whenever it converges.
To give an example, consider the series $\sum_{n >0} (-1)^n$. This series does not converge,
but we can write, for any positive integer $N$,
$$
\sum_{N>n>0}{(-1)^n}=-\frac{(-1)^N}{2}-\frac{1}{2}.
$$
Hence, we set the regularised value of this series to be $-1/2$. The details of this regularisation process is
discussed in the next section (Section 2), which is of independent interest and its other important implications
are explored in our forthcoming article \cite{MS}. Although, we need to build on the work of the second author
\cite{BS2}, there are key differences that we encounter here. 
For example, together with the Euler-Maclaurin summation formula, we need an extension of the Euler-Boole summation formula.
A structural proof of the Euler-Boole summation formula along with a proof of the Euler-Maclaurin summation formula
was presented in \cite{BCM-1} by Borwein, Calkin and Manna.
We follow their exposition to get a generalised Euler-Boole summation formula for the roots of unity.
This does not seem to have been done before. We have added it as an appendix to this article to keep our exposition
independent to the extent possible.
Finally, we give the proof of
\thmref{ad-dom} in Section 4. We remark here that even though we have convergence of the series
\eqref{mp-der} in $V_r({\bf z}) \cap \Z^r$, the holomorphicity of $\mathrm{Li}_{(z_1,\ldots,z_r)}(s_1,\ldots,s_r)$
is a more delicate question as one may not have uniform convergence of \eqref{mp}
in the domain $V_r({\bf z})$. The question of holomorphicity of $\mathrm{Li}_{(z_1,\ldots,z_r)}(s_1,\ldots,s_r)$
is also explored in \cite{MS}, by means
of understanding the local behaviour of the multiple polylogarithms at integral points.

It is important to mention that a structured study of the multiple polylogarithms was initiated by
Goncharov in 1990's. A detailed account of that can be found in \cite{ABG-2} and the
references therein. As we develop the regularisation process of the multiple polylogarithms
at the roots of unity, we refer the reader to \cite{ABG-1,BBBL,JZ-4} for some existing literature on the
special values of the multiple polylogarithms at the roots of unity. Various analytic properties of the multiple
polylogarithms, including the analytic continuation as functions of $z_1,\ldots,z_r$, have also been
studied (see \cite[\S2]{ABG-2}, \cite{JZ-2}).

\section{Regularisation process for the special values of the multiple polylogarithms and the higher order derivatives}
To define a regularisation process for the special values of the multiple zeta functions, in \cite{BS2}, the second author
needed the concept of a comparison scale and asymptotic expansion of a sequence of complex numbers
relative to the given comparison scale. Here we need the concept of asymptotic expansion of a sequence
of complex numbers relative to the given comparison scale {\it with variable coefficients}, as given in
\cite[Chap. V, \S2.5]{NB}. 

In what follows, $\N$ denotes the set of non-negative integers. Let $\mathcal{E}$ be the comparison scale,
on the set $\N$, filtered by the Fr\'echet filter, formed by the sequences 
$$
\left( (\log n)^ln^{-m} \right)_{n \geq 1},
$$
where $l \in \N$ and $m \in \Z$. Also, let $r$ be a positive integer and $z_1,\ldots,z_r \in \C$ be roots of unity.
Let $\mathcal{C}$ be the $\C$-algebra generated by the constant sequence $(1)_{n\geq 1}$ and the sequences
of the form $(z_i^n)_{n\geq 1}$ for $1 \le i \le r$.
Note that as a $\C$-vector space, $\mathcal C$ is finite dimensional since the set
$$
\mathcal S:=\{z_1^{k_1} \cdots z_r^{k_r}: k_i \in \N \text{ for } 1 \le i \le r\}
$$
is finite. For every element $\xi \in \mathcal S$, there are infinitely many $(k_1, \ldots, k_r) \in \N^r$
such that $\xi=z_1^{k_1} \cdots z_r^{k_r}$. By the exponent of $\xi$, we mean the smallest such
element in $\N^r$, as per the dictionary ordering. Let $S$ denote the set of exponents of $\mathcal S$.
Then $\mathcal C$ has a basis of the form
\begin{equation}\label{basis}
\mathcal B=\{ (z_1^{k_1 n} \cdots z_r^{k_r n})_{n \ge 1} : (k_1, \ldots, k_r) \in S\},
\end{equation}
As $(1)_{n \ge 1} \in \mathcal B$, we have $(0, \ldots,0) \in S$. It is immediate that $\mathcal B$
is a spanning set for $\mathcal C$. Moreover,
as we are only taking $(k_1, \ldots, k_r) \in S$, we get the linear independence
using a Vandermonde determinant trick. The $\C$-algebra $\mathcal{C}$ satisfies the following conditions:
\begin{enumerate}

\item[(a)] for every sequence $(a_n)_{n\geq 1} \in \mathcal{C}$, $(a_n)_{n\geq 1}$ is a bounded sequence;

\item[(b)] if $(a_n)_{n\geq 1} \in \mathcal{C}$ with $a_n=o(1)$ for sufficiently large $n$,
then we must have $a_n=0$ for all but finitely many $n$.

    
\end{enumerate}
This allows us to define the concept of asymptotic expansion of a sequence
of complex numbers relative to $\mathcal E$ with coefficients in $\mathcal C$ (see \cite[Chap. V, \S2.5]{NB}).

\begin{defn}\label{def-as-exp}
We say that a complex sequence $(u_n)_{n \geq 1}$ has an asymptotic expansion relative
to the comparison scale $\mathcal{E}$ with coefficients in $\mathcal{C}$ to arbitrary precision if 
there exists a family of sequences $\left( (a_n^{(l,m)})_{n\geq 1}\right)_{l\in\N,m\in\Z}$ in $\mathcal{C}$
such that we can write
$$
u_n=\sum_{l \geq 0, \hspace{0.5mm} m \leq A}a_{n}^{(l,m)}(\log n)^ln^{-m}+o(n^{-A}) \text{ as } n \to\infty,
$$
for any integer $A$, where for $m$ sufficiently small, these sequences are the constant sequence $(0)_{n\geq 1}$,
and for any $m$, these sequences are the constant sequence $(0)_{n\geq 1}$ for all but finitely many $l \in \N$.
The smallest $m$ such that a sequence of the form $\left(a_n^{(l,m)}\right)_{n\geq 1}$ is non-zero for some $l \in \N$, is called
the order of the asymptotic expansion.
\end{defn}

If a complex sequence $(u_n)_{n \geq 1}$ has an asymptotic expansion relative to $\mathcal{E}$
with coefficients in $\mathcal{C}$ to arbitrary precision, then this expansion is unique
(see \cite{NB}, Chap $V$, \S2.5). Now we consider the sequence $\big(a_{n}^{(0,0)}\big)_{n \ge 1}$ and we can
write it as a linear combination of elements of $\mathcal B$ as follows:
$$
\left(a_{n}^{(0,0)}\right)_{n \ge 1}=\sum_{(k_1,\ldots,k_r) \in S} \lambda_{(k_1,\ldots,k_r)} (z_1^{k_1 n} \cdots z_r^{k_r n})_{n \ge 1},
$$
where $\lambda_{(k_1,\ldots,k_r)} \in \C$.

\begin{defn}\label{def-reg-val}
Let $(u_n)_{n \geq 1}$ be a complex sequence having an asymptotic expansion relative to $\mathcal{E}$
with coefficients in $\mathcal{C}$ to arbitrary precision. The complex number $\lambda_{(0,\ldots,0)}$
is called the regularised value of the sequence $(u_n)_{n \geq 1}$ (relative to $\mathcal E$ and $\mathcal C$).
\end{defn}

Note that if two sequences differ by only finitely many terms and one of them has
an asymptotic expansion relative to $\mathcal E$ with coefficients in $\mathcal{C}$ to arbitrary precision,
then the other one also has such an expansion and their asymptotic expansions are the same. This
observation allows us to extend Definition \ref{def-reg-val} to sequences $(u_n)$ which are only
defined for $n$ large enough.

\begin{ex}\rm
Let $\mathcal C_1$ be the $\C$-algebra generated by the constant sequence $(1)_{n\geq 1}$ and sequence $((-1)^n)_{n\geq 1}$.
For the sequences, $(u_n)_{n\ge 1}=\sum_{n>m>0}{(-1)^m}$ and $(v_n)_{n\ge 1}=\sum_{n>m>0}{(-1)^m \log m}$, we have
$$
u_n=-\frac{(-1)^n}{2}-\frac{1}{2} \ \text{ and } \ v_n= -\frac{(-1)^n \log n}{2} + \frac{\log(\pi/2)}{2} + o(1), \text{ as } n \to \infty.
$$
Hence the regularised value of the sequences $(u_n)_{n \geq 1}$ and
$(v_n)_{n \geq 1}$ (relative to $\mathcal E$ and $\mathcal C_1$) are $-1/2$ and $\frac{\log(\pi/2)}{2}$, respectively.
\end{ex}

These examples are in the spirit of the following proposition that we need in the context of the regularisation of the
special values of the multiple polylogarithms. The proposition below can be seen as an extension of \cite[Proposition 1]{BS2}.

\begin{prop}\label{partial sum}
Let $(u_n)_{n\geq1}$ be a sequence of complex numbers having an asymptotic expansion
to arbitrary precision relative to $\mathcal{E}$ with coefficients in $\mathcal{C}$.
Then the sequence $(v_n)_{n\geq1}$ defined by $ v_n:=\sum_{n>m>0}u_m $ also has
an asymptotic expansion to arbitrary precision relative to $\mathcal{E}$ with coefficients in $\mathcal{C}$.
\end{prop}

\begin{proof}
It is enough to show that the sequence $(v_n)_{n \ge 1}$ has an asymptotic expansion
to precision $n^{-A}$ relative to $\mathcal{E}$ with coefficients in $\mathcal{C}$, for any integer $A \ge 1$.
For $(u_n)_{n\geq1}$, we have a family of sequences $\left( (a_n^{(l,m)})_{n\geq 1}\right)_{l\in\N,m\in\Z}$
in $\mathcal{C}$ such that 
$$
u_n=\sum_{l\geq0, \hspace{0.5mm} m\leq A+1}a_n^{(l,m)}(\log n)^ln^{-m}+o(n^{-A-1}),
$$ 
as $n$ tends to $\infty$, for every $A\in\Z$. It is therefore enough to establish a 
suitable asymptotic expansion for $(v_n)_{n \ge 1}$ in the following three cases:
\begin{enumerate}
\item $u_n=(\log n)^ln^{-m}$ for $l\in\N,m\in\Z$;
\item $u_n=o(n^{-A-1})$ as $n\rightarrow\infty$;
\item $u_n=z^n(\log n)^ln^{-m}$ for $l\in\N,m\in\Z$ and $z=z_1^{k_1}\cdots z_r^{k_r}$ for some $k_1,\ldots,k_r \in \N$.
\end{enumerate}

Cases (1) and (2) follow from \cite[Proposition 1]{BS2}.
In case (3), we apply the generalised Euler-Boole summation formula (see \eqref{MFormula} in Appendix A).
Hence, we just need to note that the derivatives of the functions on $(1, \infty)$ of the form
$f_{(l,m)}(t)=(\log t)^lt^{-m}$, for $l\in\N,m\in\Z$, has an asymptotic expansion to the precision $n^{-A}$
relative to $\mathcal{E}$. This completes the proof.
\end{proof}

Now as an immediate corollary, we derive the following theorem.

\begin{thm}\label{Aysm}
For any $a_1,\ldots, a_r \in \mathbb{Z}$ and $k_1,\ldots, k_r\in \mathbb{N}$, the sequence $(u_n)_{n > r}$, where
$$
u_n:=\sum_{n>n_1>\cdots>n_r>0}\frac{z_1^{n_1}(\log n_1)^{k_1}\cdots z_r^{n_r}(\log n_r)^{k_r}}{n_1^{a_1}\cdots n_r^{a_r}},
$$
has an asymptotic expansion to the arbitrary precision relative to $\mathcal{E}$ with coefficients in $\mathcal{C}$.
\end{thm}

\begin{proof}
The proof follows exactly as in \cite[Theorem 1]{BS2} and hence omitted.
\end{proof}

Hence from the above theorem, as $n$ tends to $\infty$, we can write the sum 
$$
\sum_{n>n_1>\cdots>n_r>0}\frac{z_1^{n_1}(\log n_1)^{k_1}\cdots z_r^{n_r}(\log n_r)^{k_r}}{n_1^{a_1}\cdots n_r^{a_r}}
=\sum_{l, \hspace{0.5mm} m \ge 0}a_{n}^{(l,m)}(\log n)^l n^{m}+o(1),
$$
where $(a_{n}^{(l,m)})_{n>r}\in \mathcal{C}$ is such that $(a_{n}^{(l,m)})_{n>r}$ is the constant sequence $(0)_{n >r}$
for all but finitely many pairs $(l,m)$.

\begin{defn}
Let $r\geq0$ be an integer. For any $(a_1,\ldots,a_r)\in \Z^r$ and $(k_1,\ldots,k_r)\in \N^r$,
the regularised value of the sequence $(u_n)_{n > r}$, where
$$
u_n:=\sum_{n>n_1>\cdots>n_r>0} \frac{z_1^{n_1}(\log n_1)^{k_1}\cdots z_r^{n_r}(\log n_r)^{k_r}}{n_1^{a_1}\cdots n_r^{a_r}},
$$
is denoted by $\ell_{[k_1,\ldots,k_r]}^{(a_1,\ldots,a_r)}(\mathbf{z})$ and we call it the multiple Stieltjes constant for the multiple polylogarithms
at $(a_1,\ldots,a_r)$ of order $[k_1,\ldots, k_r]$.
\end{defn}

\section{Proof of \thmref{big-dom}}
For $r\geq 1$ an integer and ${\bf z}:=(z_1,\ldots,z_r) \in \C^r$  where $|z_i|\leq1$ for all $1\leq i\leq r$,
we recall that $U_r(\mathbf{z})$ is defined in \eqref{Urz} as follows:
\begin{equation*}
\begin{split}
U_r(\mathbf{z}):=\{ & (s_1,\ldots,s_r)\in\mathbb{C}^r: \Re(s_1+\cdots+s_i)>i \text{ if } 1\leq i < q({\bf z}) \\
&\text{ and } \Re(s_1+\cdots+s_i)>i-1 \text{ if } q({\bf z})\leq i \leq r \} .
\end{split}
\end{equation*}
For an integer $N \ge 1$ and $(a_1,\ldots,a_r) \in U_r(\mathbf{z})$,  we denote the partial sum 
$$
\sum_{N>n_1>\cdots>n_r>0}\frac{z_1^{n_1}\cdots z_r^{n_r}}{n_1^{a_1}\cdots n_r^{a_r}}
$$
by $t_{(z_1,\ldots,z_r)}(a_1,\ldots,a_r)_N$ with the convention that for $N \le r$, we have $t_{(z_1,\ldots,z_r)}(a_1,\ldots,a_r)_N=0$.
To show the series in \eqref{mp} converges, it is therefore enough to show that
the difference $|t_{(z_1,\ldots,z_r)}(a_1,\ldots,a_r)_M - t_{(z_1,\ldots,z_r)}(a_1,\ldots,a_r)_N| \to 0$ as $N \to \infty$
for any $M>N$. We in fact, prove a stronger result.
Note that
$$
t_{(z_1,\ldots,z_r)}(a_1,\ldots,a_r)_M - t_{(z_1,\ldots,z_r)}(a_1,\ldots,a_r)_N=
\sum_{M>n_1>\cdots>n_r>0 \atop n_1 \ge N}\frac{z_1^{n_1}\cdots z_r^{n_r}}{n_1^{a_1}\cdots n_r^{a_r}},
$$
and we denote the series on the right-hand side by $t_{(z_1,\ldots,z_r)}(a_1,\ldots,a_r)_{M,N}$.
\thmref{big-dom} then follows from the following statement.

\begin{prop}\label{prop-Urz}
There is a polydisc $D$ around $(a_1,\ldots,a_r) \in U_r(\mathbf{z})$
and $\epsilon >0$ such that
$$
\|t_{(z_1,\ldots,z_r)}(s_1,\ldots,s_r)_{M,N}\|_{\overline D}=O(N^{-\epsilon}),
$$
as $N \to \infty$. Here $\overline D$ denotes the closure of $D$ in $\C^r$ and $\|t_{(z_1,\ldots,z_r)}(s_1,\ldots,s_r)_{M,N}\|_{\overline D}$
denotes the supremum of $|t_{(z_1,\ldots,z_r)}(s_1,\ldots,s_r)_{M,N}|$ on $\overline D$.
\end{prop}

Here by a polydisc around a point $(b_1,\ldots,b_r) \in \C^r$, we mean an open subset of $\C^r$
which is obtained by taking the Cartesian products of open discs around each $b_i$ in $\C$.
We now recall the definition of normal convergence of a series of functions.

\begin{defn}\label{def-nc-1}
Let $X$ be a set and $(f_i)_{i \in I}$ be a 
family of complex valued functions
defined on $X$. We say that the family of 
functions $(f_i)_{i \in I}$
is normally summable on $X$ or the series $\sum_{i \in I} f_i$ 
converges normally on $X$ if 
$$
\|f_i\|_X := \sup_{x \in X} |f(x)| < \infty ,
~\text{      for all  }i \in I,
$$ 
and the family of real numbers 
$(\| f_i \|_X)_{i \in I}$ is summable. 
\end{defn}

We need the following lemma from \cite[Proposition 3]{BS1}.

\begin{lem}\label{normal}
The family of functions
$$
\left(\frac{(s_1-1)_{k+1}}{(k+1)!}\frac{z_1^{n_1} \cdots z_r^{n_r}}{n_1^{s_1+k}\cdots n_r^{s_r}}\right)_{n_1>\cdots >n_r>0, \hspace{0.5mm} k\geq 0}
$$
is normally summable on any compact subset of $U_r$. Here $(s_1-1)_{k+1}$ denotes the product
$(s_1-1)(s_1)\cdots(s_1+k-1)$.
\end{lem}

We also need the following lemma.

\begin{lem}\label{tail-normal}
Let  $(b_1,\ldots, b_r)\in \C^r$. Let $k_0\in \N$ be the least non-negative integer
such that $(b_1+k_0,b_2,\ldots, b_r)\in U_r$. Then there exists a polydisc $D$ around $(b_1,\ldots, b_r)$
and $\epsilon>0$ such that the family of functions 
$$
\left(\frac{(s_1-1)_{k+1}}{(k+1)!}\frac{z_1^{n_1} \cdots z_r^{n_r}}{n_1^{s_1+k}\cdots n_r^{s_r}}\right)_{\substack {n_1>\cdots > n_r>0 \\ n_1\geq N\geq2;\hspace{0.5mm} k\geq k_0}}  
$$
is normally summable on the closure $\overline{D}$ of $D$ in $\C^r$ and its sum is $O(N^{-\epsilon})$ as $N$ tends to $\infty$.
\end{lem}

\begin{proof} As $|z_i|\le1$ for all $1\leq i\leq r$, the proof follows from \cite[Lemma 2]{BS2}.
\end{proof}

Our next lemma is key to the proof of \propref{prop-Urz}. It can be viewed as an 
extension of \cite[Theorems 6, 7]{BS1} for the partial tails of the series in \eqref{mp}.

\begin{lem}\label{lem-trans-tail}
Let $r\geq 1$ be an integer and $(s_1,\ldots,s_r)\in \C^r$. Then  for $M>N\geq 2$, the following equality holds:
For $r=1$, we have
\begin{align}\label{trans-tail-1}
(z_1-1)t_{z_1}(s_1-1)_{M,N}+ \frac{z_1^N}{(N-1)^{s_1-1}} - \frac{z_1^M}{(M-1)^{s_1-1}} 
=\sum_{k\geq0}\frac{(s_1-1)_{k+1}}{(k+1)!} t_{z_1}(s_1+k)_{M,N},
\end{align}
and for $r>1$, we have
\begin{align}\label{trans-tail-r}
\begin{split}
&z_1 t_{(z_{[1,2]},z_3, \ldots, z_r)}(s_1+s_2-1,s_3,\ldots,s_r)_{M-1,N} +
(z_1-1)t_{(z_1, \ldots, z_r)}(s_1-1,s_2,\ldots,s_r)_{M,N}\\
&+\frac{z_1^N}{(N-1)^{s_1-1}}t_{(z_2, \ldots, z_r)}(s_2,\ldots,s_r)_{N}
-\frac{z_1^M}{(M-1)^{s_1-1}}t_{(z_2, \ldots, z_r)}(s_2,\ldots,s_r)_{M-1}\\
&=\sum_{k\geq0}\frac{(s_1-1)_{k+1}}{(k+1)!}t_{(z_1, \ldots, z_r)}(s_1+k,s_2,\ldots,s_r)_{M,N}.
\end{split}
\end{align}
\end{lem}

\begin{proof}
We will start with the following series expansion for any complex number $s_1$ and integer $n_1\geq2$:
$$
(n_1-1)^{1-s_1}-n_1^{1-s_1}=\sum_{k \geq 0}\frac{(s_1-1)_{k+1}}{(k+1)!}n_1^{-s_1-k}.
$$
Multiplying both sides by $\frac{z_1^{n_1}z_2^{n_2}\cdots z_r^{n_r}}{n_2^{s_2}\cdots n_r^{s_r}}$
and summing for $n_1>\cdots>n_r>0$ with $M>n_1\geq N \ge 2$, we get
\begin{align}\label{3}
\begin{split}
& \sum_{\substack{ M>n_1>\cdots>n_r>0\\ n_1\geq N}}\left( \frac{z_1^{n_1}}{(n_1-1)^{s_1-1}} -\frac{z_1^{n_1}}{n_1^{s_1-1}}\right)\frac{z_2^{n_2}\cdots z_r^{n_r}}{n_2^{s_2}\cdots n_r^{s_r}}\\
&= \sum_{\substack{ M>n_1>\cdots>n_r>0\\ n_1\geq N}}\sum_{k \geq 0}\frac{(s_1-1)_{k+1}}{(k+1)!}\frac{z_1^{n_1}\cdots z_r^{n_r}}{n_1^{s_1+k}\cdots n_r^{s_r}}.
\end{split}
 \end{align}
Note that $\sum_{k\geq0}\frac{(s_1-1)_{k+1}}{(k+1)! \hspace{0.5mm} n_1^k}$ converges (using the ratio test).
Hence interchanging the summations we get that the right-hand side of \eqref{3} is same as
$$
\sum_{k\geq0}\frac{(s_1-1)_{k+1}}{(k+1)!}t_{(z_1, \ldots, z_r)}(s_1+k,s_2,\ldots,s_r)_{M,N}.
$$
When $r=1$, for the left-hand side of \eqref{3} we note that
\begin{align*}
&\sum_{M>n_1\geq N}\left( \frac{z_1^{n_1}}{\left(n_1-1\right)^{s_1-1}}-\frac{z_1^{n_1}}{n_1^{s_1-1}} \right)\\
                    &= \frac{z_1^{N}}{(N-1)^{s_1-1}} - \frac{z_1^{N}}{N^{s_1-1}} + \frac{z_1^{N+1}}{N^{s_1-1}} - \frac{z_1^{N+1}}{(N+1)^{s_1-1}} + \cdots
                    +\frac{z_1^{M-1}}{(M-2)^{s_1-1}} - \frac{z_1^{M-1}}{(M-1)^{s_1-1}}.
\end{align*}
We add and subtract $\frac{z_1^M}{(M-1)^{s_1-1}}$ on the right-hand side to get                     
\begin{align*}                  
\sum_{M>n_1\geq N}\left( \frac{z_1^{n_1}}{\left(n_1-1\right)^{s_1-1}}-\frac{z_1^{n_1}}{n_1^{s_1-1}} \right)=\frac{z_1^{N}}{(N-1)^{s_1-1}} + (z_1-1)\sum_{M>n_1\geq N}\frac{z_1^{n_1}}{ n_1^{s_1-1}}-\frac{z_1^{M}}{(M-1)^{s_1-1}},              
\end{align*}
which is the left-hand side of \eqref{trans-tail-1}.
When $r>1$, for the left-hand side of equation \eqref{3} we write
\begin{align*}
 &\sum_{\substack{ M>n_1>\cdots>n_r>0\\ n_1\geq N}}\left( \frac{z_1^{n_1}}{(n_1-1)^{s_1-1}} -\frac{z_1^{n_1}}{n_1^{s_1-1}}\right)\frac{z_2^{n_2}\cdots z_r^{n_r}}{n_2^{s_2}\cdots n_r^{s_r}}\\
                    &=\sum_{\substack{ M-1>n_2>\cdots>n_r>0\\ n_2\geq N}}\sum_{n_1=n_2+1}^{M-1}\left( \frac{z_1^{n_1}}{(n_1-1)^{s_1-1}} -\frac{z_1^{n_1}}{n_1^{s_1-1}}\right)\frac{z_2^{n_2}\cdots z_r^{n_r}}{n_2^{s_2}\cdots n_r^{s_r}}\\
                    & \ \ +\sum_{\substack{ M-1>n_2>\cdots>n_r>0\\ n_2 < N}}\sum_{n_1=N}^{M-1}\left( \frac{z_1^{n_1}}{(n_1-1)^{s_1-1}} -\frac{z_1^{n_1}}{n_1^{s_1-1}}\right)\frac{z_2^{n_2}\cdots z_r^{n_r}}{n_2^{s_2}\cdots n_r^{s_r}}.
\end{align*}
Computing the first sum in the right-hand side of the above expression we get
\begin{align}\label{sum1}
\begin{split}
                  &\sum_{\substack{ M-1>n_2>\cdots>n_r>0\\ n_2\geq N}}\left(\frac{z_1^{n_2+1}}{n_2^{s_1-1}} + (z_1-1)\sum_{M>n_1\geq n_2+1}\frac{z_1^{n_1}}{ n_1^{s_1-1}}-\frac{z_1^{M}}{(M-1)^{s_1-1}}\right)\frac{z_2^{n_2}\cdots z_r^{n_r}}{n_2^{s_2}\cdots n_r^{s_r}}\\
                 &=z_1\sum_{\substack{ M-1>n_2>\cdots>n_r>0\\ n_2\geq N}}\frac{z_{[1,2]}^{n_2}z_3^{n_3}\cdots z_r^{n_r}}{n_2^{s_1+s_2-1}n_3^{s_3}\cdots n_r^{s_r}}\\
                 & \ \ +(z_1-1)\sum_{\substack{ M-1>n_2>\cdots>n_r>0\\ n_2\geq N}}\left(\sum_{M>n_1\geq n_2+1}\frac{z_1^{n_1}}{ n_1^{s_1-1}}\right)\frac{z_2^{n_2}\cdots z_r^{n_r}}{n_2^{s_2}\cdots n_r^{s_r}}\\
               & \ \ -\frac{z_{1}^{M}}{(M-1)^{s_1-1}}\sum_{\substack{ M-1>n_2>\cdots>n_r>0\\ n_2\geq N}}\frac{z_2^{n_2}\cdots z_r^{n_r}}{n_2^{s_2}\cdots n_r^{s_r}}.
 \end{split}
\end{align}
Similarly, the same computation can be done for the second sum, and we get
\begin{align}\label{sum2}
\begin{split}
                  &\sum_{\substack{ M-1>n_2>\cdots>n_r>0\\ n_2 < N}}\left(\frac{z_1^{N}}{(N-1)^{s_1-1}} + (z_1-1)\sum_{M>n_1\geq N}\frac{z_1^{n_1}}{ n_1^{s_1-1}}-\frac{z_1^{M}}{(M-1)^{s_1-1}}\right)\frac{z_2^{n_2}\cdots z_r^{n_r}}{n_2^{s_2}\cdots n_r^{s_r}}\\
                 &=\frac{z_{1}^{N}}{(N-1)^{s_1-1}}\sum_{\substack{ M-1>n_2>\cdots>n_r>0\\ n_2 < N}}\frac{ z_2^{n_2}\cdots z_r^{n_r}}{n_2^{s_2}\cdots n_r^{s_r}}\\
                 &\ \ +(z_1-1)\sum_{\substack{ M-1>n_2>\cdots>n_r>0\\ n_2 < N}}\left(\sum_{M>n_1\geq N}\frac{z_1^{n_1}}{ n_1^{s_1-1}}\right)\frac{z_2^{n_2}\cdots z_r^{n_r}}{n_2^{s_2}\cdots n_r^{s_r}}\\
               &\ \ -\frac{z_{1}^{M}}{(M-1)^{s_1-1}}\sum_{\substack{ M-1>n_2>\cdots>n_r>0\\ n_2 < N}}\frac{z_2^{n_2}\cdots z_r^{n_r}}{n_2^{s_2}\cdots n_r^{s_r}}.
 \end{split}
\end{align}
Combining the middle terms on the right-hand side of \eqref{sum1} and \eqref{sum2}, we get 
 \begin{align*}
&(z_1-1)\sum_{\substack{ M-1>n_2>\cdots>n_r>0\\ n_2\geq N}}\left(\sum_{M>n_1\geq n_2+1}\frac{z_1^{n_1}}{ n_1^{s_1-1}}\right)\frac{z_2^{n_2}\cdots z_r^{n_r}}{n_2^{s_2}\cdots n_r^{s_r}}\\
& +(z_1-1)\sum_{\substack{ M-1>n_2>\cdots>n_r>0\\ n_2 < N}}\left(\sum_{M>n_1\geq N}\frac{z_1^{n_1}}{ n_1^{s_1-1}}\right)\frac{z_2^{n_2}\cdots z_r^{n_r}}{n_2^{s_2}\cdots n_r^{s_r}}\\
&=(z_1-1)\sum_{\substack{ M>n_1>\cdots>n_r>0\\ n_1\geq N}}\frac{z_1^{n_1}z_2^{n_2}\cdots z_r^{n_r}}{n_1^{s_1-1}n_2^{s_2}\cdots n_r^{s_r}}.
 \end{align*}
Similarly, combining the last terms on the right-hand side of  \eqref{sum1} and \eqref{sum2}, we get
\begin{align*}
&\frac{z_{1}^{M}}{(M-1)^{s_1-1}}\sum_{\substack{ M-1>n_2>\cdots>n_r>0\\ n_2\geq N}}\frac{z_2^{n_2}\cdots z_r^{n_r}}{n_2^{s_2}\cdots n_r^{s_r}}+\frac{z_{1}^{M}}{(M-1)^{s_1-1}}\sum_{\substack{ M-1>n_2>\cdots>n_r>0\\ n_2 < N}}\frac{z_2^{n_2}\cdots z_r^{n_r}}{n_2^{s_2}\cdots n_r^{s_r}}\\
&=\frac{z_{1}^{M}}{(M-1)^{s_1-1}}\sum_{\substack{ M-1>n_2>\cdots>n_r>0}}\frac{z_2^{n_2}\cdots z_r^{n_r}}{n_2^{s_2}\cdots n_r^{s_r}}.
\end{align*}
Hence, collecting all the terms, the left-hand side of equation \eqref{3} becomes
\begin{align*}
       &z_1\sum_{\substack{ M-1>n_2>\cdots>n_r>0\\ n_2\geq N}}\frac{z_{[1,2]}^{n_2}z_3^{n_3}\cdots z_r^{n_r}}{n_2^{s_1+s_2-1}n_3^{s_3}\cdots n_r^{s_r}}+(z_1-1)\sum_{\substack{ M>n_1>\cdots>n_r>0\\ n_1\geq N}}\frac{z_1^{n_1}z_2^{n_2}\cdots z_r^{n_r}}{n_1^{s_1-1}n_2^{s_2}\cdots n_r^{s_r}}\\
      & + \frac{z_{1}^{N}}{(N-1)^{s_1-1}}\sum_{N>n_2>\cdots>n_r>0}\frac{ z_2^{n_2}\cdots z_r^{n_r}}{n_2^{s_2}\cdots n_r^{s_r}} - \frac{z_{1}^{M}}{(M-1)^{s_1-1}}\sum_{\substack{ M-1>n_2>\cdots>n_r>0}}\frac{z_2^{n_2}\cdots z_r^{n_r}}{n_2^{s_2}\cdots n_r^{s_r}},
\end{align*}
which is the left-hand side of \eqref{trans-tail-r}. This completes the proof.
\end{proof}

Setting
$$\delta_i=\begin{cases}
1 & \text{ if } z_i \neq 1,\\
0 & \text{ if } z_i = 1,
\end{cases}
$$
we can rewrite \eqref{trans-tail-1} and \eqref{trans-tail-r} in the following (combined) form which will be easier for us to use.
We have
\begin{align}\label{trans-tail-gen}
\begin{split}
&z_1 (1-[1/r]) t_{(z_{[1,2]},z_3, \ldots, z_r)}(s_1+s_2+\delta_1-1,s_3,\ldots,s_r)_{M-1,N} \\
&+ (z_1-1)t_{(z_1, \ldots, z_r)}(s_1+\delta_1-1,s_2,\ldots,s_r)_{M,N}\\
&+\frac{z_1^N}{(N-1)^{s_1+\delta_1-1}}t_{(z_2, \ldots, z_r)}(s_2,\ldots,s_r)_{N}
-\frac{z_1^M}{(M-1)^{s_1+\delta_1-1}}t_{(z_2, \ldots, z_r)}(s_2,\ldots,s_r)_{M-1}\\
&=\sum_{k\geq0}\frac{(s_1+\delta_1-1)_{k+1}}{(k+1)!}t_{(z_1, \ldots, z_r)}(s_1+\delta_1+k,s_2,\ldots,s_r)_{M,N}.
\end{split}
\end{align}
For the proof of \propref{prop-Urz}, we also need the following corollary of Lemma \ref{lem-trans-tail}.

\begin{cor}\label{cor-lem}
Let $r, M \ge 2$ be integers. For $(a_1,\ldots,a_r) \in U_r(\mathbf{z})$,
we have a polydisc $D$ around $(a_1,\ldots,a_r)$ and $\epsilon >0$ such that as $M  \to \infty$,
$$
\left\|\frac{z_1^M}{(M-1)^{s_1+\delta_1-1}}t_{(z_2,\ldots,z_r)}(s_2,\ldots,s_r)_{M}\right\|_{\overline D}=O(M^{-\epsilon}).
$$ 
\end{cor}

\begin{proof}
We prove the result by induction on $r$. For $r=2$, consider the identity \eqref{trans-tail-1} with $N=2$
and for the variable $s_2$:
$$
(z_2-1) t_{z_2}(s_2+\delta_2-1)_{M,2} + z_2^2 - \frac{z_2^M}{(M-1)^{s_2+\delta_2-1}}
=\sum_{k \ge 0}\frac{(s_2+\delta_2-1)_{k+1}}{(k+1)!} t_{z_2}(s_2+\delta_2+k)_{M,2}.
$$
We multiply the above equation by $\frac{z_1^M}{(M-1)^{s_1+\delta_1-1}}$ to have
\begin{align*}
\begin{split}
&\frac{z_1^M(z_2-1)}{(M-1)^{s_1+\delta_1-1}} t_{z_2}(s_2+\delta_2-1)_{M,2} + \frac{z_1^Mz_2^2}{(M-1)^{s_1+\delta_1-1}}
- \frac{z_1^Mz_2^M}{(M-1)^{s_1+s_2+\delta_1+\delta_2-2}}\\
&=\sum_{k \ge 0}\frac{(s_2+\delta_2-1)_{k+1}}{(k+1)!} \frac{z_1^M}{(M-1)^{s_1+\delta_1-1}} t_{z_2}(s_2+\delta_2+k)_{M,2}.
\end{split}
\end{align*}
Note that as $(a_1,a_2) \in U_2(z_1,z_2)$, on a (small) polydisc $D$ around $(a_1,a_2)$ and 
for some suitable $\epsilon>0$, we have $\Re(s_1)+\delta_1>1+\epsilon$ and
$\Re(s_1+s_2)+\delta_1+\delta_2>2+\epsilon$ for all $(s_1,s_2) \in D$. Hence, we have $\epsilon>0$ such that
$$
\left \| \frac{z_1^M z_2^2}{(M-1)^{s_1+\delta_1-1}} \right \|_{\overline D},
\left \| \frac{z_1^Mz_2^M}{(M-1)^{s_1+s_2+\delta_1+\delta_2-2}}\right \|_{\overline D} \ll \frac{1}{M^\epsilon}.
$$
Hence, for the same $\epsilon>0$, by \lemref{normal} we have
\begin{align*}
\begin{split}
& \left\|\frac{z_1^M}{(M-1)^{s_1+\delta_1-1}}t_{z_2}(s_2)_{M,2}\right\|_{\overline D} \\ 
&\ll \frac{1}{M^\epsilon} + \sum_{k \ge 1-\delta_2} \left\| \frac{(s_2+\delta_2-1)_{k+1}}{(k+1)!} \frac{z_1^M}{(M-1)^{s_1+\delta_1-1}} t_{z_2}(s_2+\delta_2+k)_{M,2} \right\|_{\overline D}\\
& \ll \frac{1}{M^\epsilon} + \frac{1}{M^{\epsilon}} \sum_{k \ge 1-\delta_2} \sum_{M> n_2 \ge 2}\left\| \frac{(|s_2|+\delta_2+1)_{k+1}}{(k+1)!}
\frac{1}{(M-1)^{\Re(s_1)+\delta_1-1-\epsilon}} \frac{1}{n_2^{\Re(s_2)+\delta_2+k}} \right\|_{\overline D}\\
& \ll \frac{1}{M^\epsilon}+ \frac{1}{M^{\epsilon}} \sum_{k \ge 1-\delta_2} \sum_{M> n_2 \ge 2}\left\| \frac{(|s_2|+\delta_2+1)_{k+1}}{(k+1)!}
\frac{1}{n_2^{\Re(s_1+s_2)+\delta_1+\delta_2-1-\epsilon+k}} \right\|_{\overline D} \ll \frac{1}{M^\epsilon}.
\end{split}
\end{align*}
In this case we note that $t_{(z_2)}(s_2)_{M}-t_{(z_2)}(s_2)_{M,2}=z_2^2$. Hence
$$
 \left\|\frac{z_1^M}{(M-1)^{s_1+\delta_1-1}}t_{z_2}(s_2)_{M}\right\|_{\overline D} \ll \frac{1}{M^\epsilon}.
$$
Now we suppose $r>2$. Note that
$t_{(z_2,\ldots,z_r)}(s_2,\ldots,s_r)_{M}=t_{(z_2,\ldots,z_r)}(s_2,\ldots,s_r)_{M,2}$.
Here we consider the identity \eqref{trans-tail-gen} with $N=2$, for the variable $s_2,\ldots,s_r$,
and then multiply $\frac{z_1^M}{(M-1)^{s_1+\delta_1-1}}$ on both the sides to get
\begin{align}\label{trans-tail-gen-2}
\begin{split}
&\frac{z_1^Mz_2}{(M-1)^{s_1+\delta_1-1}} t_{(z_{[2,3]},z_4, \ldots, z_r)}(s_2+s_3+\delta_2-1,s_4,\ldots,s_r)_{M-1,2}\\
&+ \frac{z_1^M(z_2-1)}{(M-1)^{s_1+\delta_1-1}} t_{(z_2, \ldots, z_r)}(s_2+\delta_2-1,s_3,\ldots,s_r)_{M,2}\\
&+\frac{z_1^Mz_2^2}{(M-1)^{s_1+\delta_1-1}} t_{(z_3, \ldots, z_r)}(s_3,\ldots,s_r)_{2}
-\frac{z_1^Mz_2^M}{(M-1)^{s_1+s_2+\delta_1+\delta_2-2}}t_{(z_3, \ldots, z_r)}(s_3,\ldots,s_r)_{M-1}\\
&=\sum_{k\geq0}\frac{(s_2+\delta_2-1)_{k+1}}{(k+1)!}\frac{z_1^M}{(M-1)^{s_1+\delta_1-1}}t_{(z_2, \ldots, z_r)}(s_2+\delta_2+k,s_3,\ldots,s_r)_{M,2}.
\end{split}
\end{align}
Simple computations show that if $(a_1,\ldots,a_r) \in U_r({\bf z})$, then $(a_1,a_2+a_3+\delta_2-1,a_4, \ldots, a_r) \in U_{r-1}(z_1,z_{[2,3]},z_4,\ldots,z_r)$
and we also have $(a_1+a_2 +\delta_1+\delta_2-1,a_3, \ldots, a_r) \in U_{r-1}(z_{[1,2]},z_3,\ldots,z_r)$. This by the induction hypothesis
allows us to find polydiscs $D_1,D_2$ around the last two points, and $\epsilon_1,\epsilon_2>0$, respectively, such that
$$
\left\| \frac{z_1^Mz_2}{(M-1)^{s_1+\delta_1-1}} t_{(z_{[2,3]},z_4, \ldots, z_r)}(s_2+s_3+\delta_2-1,s_4,\ldots,s_r)_{M-1,2}\right\|_{\overline D_1} \ll \frac{1}{M^{\epsilon_1}},
$$
and
$$
\left\| \frac{z_1^Mz_2^M}{(M-1)^{s_1+s_2+\delta_1+\delta_2-2}}  t_{(z_3, \ldots, z_r)}(s_3,\ldots,s_r)_{M-1}\right\|_{\overline D_2} \ll \frac{1}{M^{\epsilon_2}}.
$$
Now we choose an $\epsilon>0$ such that $\epsilon<\epsilon_1,\epsilon_2$ and a polydisc $D$ around $(a_1,\ldots,a_r)$ such that for all
$(s_1,\ldots,s_r) \in D$, we have $\Re(s_1)+\delta_1>1+\epsilon$, $\Re(s_1+s_2)+\delta_1+\delta_2>2+\epsilon$ and moreover
$(s_1,s_2+s_3+\delta_2-1,s_4, \ldots, s_r) \in D_1$ and $(s_1+s_2 +\delta_1+\delta_2-1,s_3, \ldots, s_r) \in D_2$. For this $\epsilon>0$,
we note that
\begin{align*}
\begin{split}
&\sum_{k \ge 1-\delta_2} \left\| \frac{(s_2+\delta_2-1)_{k+1}}{(k+1)!} \frac{z_1^M}{(M-1)^{s_1+\delta_1-1}}
t_{z_2}(s_2+\delta_2+k,s_3,\ldots,s_r)_{M,2} \right\|_{\overline D}\\
& \ll \frac{1}{M^{\epsilon}} \sum_{k \ge 1-\delta_2} \sum_{M> n_2 > n_3 > \cdots > n_r>0}\left\| \frac{(|s_2|+\delta_2+1)_{k+1}}{(k+1)!}
\frac{1}{(M-1)^{\Re(s_1)+\delta_1-1-\epsilon}} \frac{1}{n_2^{\Re(s_2)+\delta_2+k}n_3^{\Re(s_3)}\cdots n_r^{\Re(s_r)}} \right\|_{\overline D}\\
& \ll \frac{1}{M^{\epsilon}} \sum_{k \ge 1-\delta_2} \sum_{M> n_2  > n_3 > \cdots > n_r>0}\left\| \frac{(|s_2|+\delta_2+1)_{k+1}}{(k+1)!}
\frac{1}{n_2^{\Re(s_1+s_2)+\delta_1+\delta_2-1-\epsilon+k}n_3^{\Re(s_3)}\cdots n_r^{\Re(s_r)}} \right\|_{\overline D} \ll \frac{1}{M^\epsilon},
\end{split}
\end{align*}
by \lemref{normal}. This, together with \eqref{trans-tail-gen-2}, implies that
$$
\left\|\frac{z_1^M}{(M-1)^{s_1+\delta_1-1}}t_{(z_2,\ldots,z_r)}(s_2,\ldots,s_r)_{M}\right\|_{\overline D}=O(M^{-\epsilon}).
$$
\end{proof}

We are now ready to prove \propref{prop-Urz}.

\begin{proof}[Proof of \propref{prop-Urz}]
We will prove the statement by induction on $r$. Let $r=1$. For this, we write the
identity \eqref{trans-tail-1} as
\begin{align*}
&(z_1-1)t_{z_1}(s_1+\delta_1-1)_{M,N}+ \frac{z_1^N}{(N-1)^{s_1+\delta_1-1}} - \frac{z_1^M}{(M-1)^{s_1+\delta-1}} \\
&=\sum_{k\geq0}\frac{(s_1-1)_{k+1}}{(k+1)!} t_{z_1}(s_1+\delta_1+k)_{M,N}.
\end{align*}
Since $a_1 \in U_1(z_1)$, we have $a_1+\delta_1>1$. Hence by \lemref{tail-normal}, there exists a disc $D$ around $a_1$ and $\epsilon>0$
such that as $N \to \infty$, the sum
$$
\sum_{k\geq 1-\delta_1}\left\|\frac{(s_1+\delta_1-1)_{k+1}}{(k+1)!} t_{z_1}(s_1+\delta_1+k)_{M,N}\right\|_{\overline D} \ll \frac{1}{N^{\epsilon}}.
$$ 
Note that choosing $D$ suitably small, we can also ensure that, as $N \to \infty$,
$$
\left\|\frac{z_1^N}{(N-1)^{s_1+\delta_1-1}} \right\|_{\overline D}, \left\|\frac{z_1^M}{(M-1)^{s_1+\delta_1-1}} \right\|_{\overline D} \ll \frac{1}{N^{\epsilon}}.
$$ 
This therefore implies that $\|t_{z_1}(s_1)_{M,N}\|_{\overline D} = O(N^{-\epsilon})$ as $N \to \infty$.
Now for $r \ge 2$, we use \eqref{trans-tail-gen}. We first note that if $(a_1,\ldots,a_r) \in U_r({\bf z})$, then
$(a_1+a_2+\delta_1-1,a_3, \ldots,a_r) \in U_{r-1}(z_{[1,2]},z_3,\ldots,z_r)$. Using the induction hypothesis, we can find
a polydisc $D_1$ around the point $(a_1+a_2+\delta_1-1,a_3, \ldots,a_r)$ and $\epsilon_1>0$ such that
$$
\left\| t_{(z_{[1,2]},z_3, \ldots, z_r)}(s_1+s_2+\delta_1-1,s_3,\ldots,s_r)_{M-1,N}\right\|_{\overline D_1} \ll \frac{1}{N^{\epsilon_1}}.
$$
Now we find a polydisc $D$ around $(a_1,\ldots,a_r)$ such that for all the points $(s_1,\ldots,s_r) \in D$, we have
$(s_1+s_2 +\delta_1-1,s_3, \ldots, s_r) \in D_1$. Choosing $D$ small enough, by \corref{cor-lem}, we can further ensure that
for some $0<\epsilon<\epsilon_1$, we have as $N \to \infty$,
$$
\left\|\frac{z_1^N}{(N-1)^{s_1+\delta_1-1}}t_{(z_2, \ldots, z_r)}(s_2,\ldots,s_r)_{N}\right\|_{\overline D},
\left\|\frac{z_1^M}{(M-1)^{s_1+\delta_1-1}}t_{(z_2, \ldots, z_r)}(s_2,\ldots,s_r)_{M-1} \right\|_{\overline D} \ll \frac{1}{N^{\epsilon}}.
$$
Moreover, by \lemref{tail-normal}, we have as $N \to \infty$,
$$
\sum_{k\geq1-\delta_1} \left\|  \frac{(s_1+\delta_1-1)_{k+1}}{(k+1)!}t_{(z_1, \ldots, z_r)}(s_1+\delta_1+k,s_2,\ldots,s_r)_{M,N}\right\|_{\overline D} \ll \frac{1}{N^{\epsilon}}.
$$
Hence using \eqref{trans-tail-gen}, we finally deduce that as $N \to \infty$,
$$
\|t_{(z_1,\ldots,z_r)}(s_1,\ldots,s_r)_{M,N}\|_{\overline D}=O(N^{-\epsilon}).
$$
This completes the proof of \propref{prop-Urz}.
\end{proof}

With this the proof of \thmref{big-dom} is complete. In the next section, we prove \thmref{ad-dom}.

\section{Proof of \thmref{ad-dom}}
For $u_n$ as in \thmref{Aysm}, we prove \thmref{ad-dom} by estimating the order of the asymptotic expansion of the sequence
$(u_n)_{n\geq 1}$ relative to $\mathcal{E}$ with coefficients in $\mathcal{C}$.
We first prove the following lemma where $z=z_1^{k_1}\cdots z_r^{k_r}\ne 1$ for some $k_1,\ldots,k_r \in \N$.
 
\begin{lem}\label{r=1}
Let $a \in \Z$ and $P(X) \in \C[X]$. Then the order of the asymptotic expansion of the sequence $(v_n)_{n\geq 1}$
relative to $\mathcal{E}$ with coefficients in $\mathcal{C}$, where the sequence $(v_n)_{n\geq 1}$ defined by
\begin{equation}\label{r1}
v_n:= \sum_{m=1}^{n-1}\frac{z^m P(\log m)}{m^a},
\end{equation}
is at least $\min(0, a)$.
\end{lem}

\begin{proof} Note that it is enough to prove the statement when $P(X)$ is a monomial. 
For $k \in \N$, the order of the asymptotic expansion of the sequence $\left(\frac{(\log n)^{k}}{n^a}\right)_{n\geq1}$ (relative to
$\mathcal{E}$ with coefficients in $\mathcal{C}$) is $a$ and for the sequence $\frac{d}{dx}\frac{(\log x)^{k}}{x^a}|_{x=n}$
is $a+1$. Hence, by the generalised Euler-Boole summation formula \eqref{MFormula} in Appedinx \ref{GEBS},
the order of the asymptotic expansion of the sequence $\left(v_n\right)_{n\geq 1}$ is at least $\min(0, a).$
\end{proof}

\begin{rmk}\label{zn1}\rm
For $v_n$ as in \eqref{r1}, by applying \lemref{r=1} and \thmref{Aysm}, we see that for any integer $A\geq a$ and $0<\epsilon<1$, we have,
as $n\rightarrow \infty$,
\begin{align*}
v_n=c+\frac{z^nP_{0}(\log n)}{n^{a}}+\frac{z^nP_{1}(\log n)}{n^{a+1}}+\cdots+\frac{z^nP_{A-a}(\log n)}{n^A}+o\left(\frac{1}{n^{A+\epsilon}}\right),
\end{align*}
where $c\in \C$ and $P_{i}(X) \in \C[X]$ for all $0\leq i \leq A-a.$
\end{rmk}

\begin{rmk}\label{Rzeta} \rm
It is worthwhile to recall \cite[Remark 6]{BS2} that for $a \in \Z$ and $Q(X) \in \C[X]$,
the order of the asymptotic expansion of the sequence $(v_n)_{n\geq 1}$ relative to $\mathcal{E}$ with coefficients in $\mathcal{C}$, where
$$
v_n:=\sum_{m=1}^{n-1}\frac{Q(\log m)}{m^a},
$$
is at least min$(0,a-1)$ and
for any integer $A\geq a-1$ and $0<\epsilon<1$, we have, as $n\rightarrow \infty$,
\begin{align*}
v_n=d+\frac{Q_{0}(\log n)}{n^{a-1}}+\frac{Q_{1}(\log n)}{n^{a}}+\cdots+\frac{Q_{A-a+1}(\log n)}{n^A}+o\left(\frac{1}{n^{A+\epsilon}}\right),
\end{align*}
where $d\in \C$ and $Q_{i}(X) \in \C[X]$  for all $0\leq i \leq A-a+1$.
\end{rmk}


To write the proof inductively, we need some notations similar to \eqref{I-Q}. For $1\le i \le j \le r$, define
$$
q_{[i,j]}=q_{[i,j]}({\bf z}):=
\begin{cases}
1 &\text{ if } z_{[i,j]}=1,\\
0 &\text{ if } z_{[i,j]}\neq 1
\end{cases}
\ \text{ and } \ Q_{[i,j]}=Q_{[i,j]}({\bf z}):=q_{[i,j]}+\cdots+q_{[j,j]}.
$$
Note that taking $i=1$, we get $Q_{[1,j]}=Q_j({\bf z})$, as in \eqref{I-Q}.
We set $Q_{[i,j]}=0$ if $1 \le j<i \le r$.
Further, for $(a_1,\ldots,a_r)\in \Z^r$, set $A_{[i,j]}:=a_i+\cdots+a_j$ for $1\leq i\leq j\leq r$.
In this context, we prove the following proposition which enables us to complete the proof of \thmref{ad-dom}.

\begin{prop}\label{Extension}
Let $r\geq 1$ be an integer. Let $(u_n)_{n\geq 1}$ be the sequence defined as in \thmref{Aysm}. Then the order of the asymptotic
expansion of the sequence $(u_n)_{n\geq 1}$, relative to $\mathcal{E}$ with coefficients in $\mathcal{C}$, has order at least
$\min \left(0, A_{[1,1]}-Q_{[1,1]},A_{[1,2]}-Q_{[1,2]},\ldots, A_{[1,r]}-Q_{[1,r]}\right)$.
More precisely, for sufficiently large integer $A$ and $0<\epsilon<1$, we have, as $n\rightarrow \infty$,
\begin{equation}\label{mp-der-exp}
u_n= c +\sum_{i=1}^{r}\sum_{j=0}^{A-A_{[1,i]}+Q_{[1,i]}}\frac{F_{(i,j)}(\log n)}{n^{A_{[1,i]}-Q_{[1,i]}+j}}z_{[1,i]}^{n}
+o\left(\frac{1}{n^{A+\epsilon}}\right),
\end{equation}
where $c\in \C$ and $F_{i,j}(X) \in \C[X]$ for all $i, j$ as above.
\end{prop}

Note that \propref{Extension} extends \cite[Remark 7]{BS2} for the series defining the multiple polylogarithms.
\begin{proof}[Proof of \propref{Extension}]
We prove this by induction on $r$. When $r=1$, the statement follows from Remarks \ref{zn1} and \ref{Rzeta}.
Now we assume that $r\geq 2$. Let $A\geq 1$ be a given positive integer. Consider the sequence $(v_{n_1})_{n_1\geq 1}$ defined by 
$$
v_{n_1}:=\sum_{n_1> n_2>\cdots>n_r>0}\frac{z_2^{n_2}\log^{k_2} n_2\cdots z_r^{n_r}\log^{k_r} n_r}{n_2^{a_2}\cdots n_r^{a_r}}.
$$
Now by the induction hypothesis, corresponding to $(z_2,\ldots,z_r)$, we get that for $A'=A+a_1-1$ and $0<\epsilon<1$, we have,
as $n_1\rightarrow \infty$,
\begin{align*}
v_{n_1}=c_1+\sum_{i=2}^{r}\sum_{j=0}^{A'+Q_{[2,i]}-A_{[2,i]}}\frac{G_{(i,j)}(\log n_1)}{n_1^{A_{[2,i]}-Q_{[2,i]}+j}}z_{[2,i]}^{n_1}+
o\left(\frac{1}{n_1^{A'+\epsilon}}\right),
\end{align*}
where $c_1 \in \C$ and $G_{(i,j)}(X) \in \C[X]$ for all $i, j$. Let $(w_{n_1})_{n_1\geq 1}$ be the sequence
$w_{n_1}=\frac{z_1^{n_1}(\log n_1)^{k_1}}{n_1^{a_1}}$. Then 
$u_n=\sum_{n_1=1}^{n-1}(v_{n_1}\cdot w_{n_1})$ where 
$$
v_{n_1}\cdot w_{n_1}=c_1\frac{z_1^{n_1}(\log n_1)^{k_1}}{n_1^{a_1}}
+\sum_{i=2}^{r}\sum_{j=0}^{A'+Q_{[2,i]}-A_{[2,i]}}\frac{(\log n_1)^{k_1}G_{(i,j)}(\log n_1)}{n_1^{A_{[2,i]}-Q_{[2,i]}+j+a_1}}
z_{[1,i]}^{n_1} + o(n_1^{-A'+a_1-\epsilon}).
$$
Note that $A_{[2,i]}+a_1=A_{[1,i]}$. Now if we write $X^{k_1}G_{(i,j)}(X)$ as $H_{(i,j)}(X)$, then
%
to get the desired asymptotic expansion of $u_n$ (relative to $\mathcal{E}$ with coefficients in $\mathcal{C}$),
we need to look at the  asymptotic expansion of the sequence of the form,
 $$
\sum_{n_1=1}^{n-1}z_{[1,i]}^{n_1}\frac{H_{(i,j)}(\log n_1)}{n_1^{A_{[1,i]}-Q_{[2,i]}+j}} 
$$
for each $i,j$ as above, together with for $(i,j)=(1,0)$ where $H_{(1,0)}(\log n_1):=c_1(\log n_1)^{k_1}$,
as $\sum_{n_1=1}^{n-1}o(n_1^{-A'+a_1-\epsilon})=c_0+o(n^{-A'+a_1-1-\epsilon})=c_0+o(n^{-A-\epsilon})$,
for some $c_0 \in \C$.

Now for each $1\leq i\leq r$, $q_{[1,i]}$ is a value that depends on the value of the product $z_{[1,i]}$,
namely, $q_{[1,i]}=0$ if $z_{[1,i]} \ne 1$ and $q_{[1,i]}=1$ if $z_{[1,i]} = 1$. So using 
Remarks \ref{zn1} and \ref{Rzeta},
we see that, as $n \to \infty$, each of the above sums is of the form
$$
c_{i,j} + z_{[1,i]}^n \sum_{k=0}^{A-A_{[1,i]}+Q_{[1,i]}-j} 
\frac{H_{(i,j,k)}(\log n)}{n^{A_{[1,i]}-Q_{[2,i]}-q_{[1,i]}+j+k}} + o(n^{-A-\epsilon}),
$$
for $c_{i,j} \in \C$ and $H_{(i,j,k)}(X) \in \C[X]$. Summing for $i,j$ and collecting the terms for $j+k=l$, we get
that, as $n \to \infty$,
$$
u_n=c+\sum_{i=1}^{r}\sum_{l=0}^{A-A_{[1,i]}+Q_{[1,i]}}
\frac{F_{(i,l)}(\log n)}{n^{A_{[1,i]}-Q_{[1,i]}+l}}z_{[1,i]}^{n}+o\left(\frac{1}{n^{A+\epsilon}}\right),
$$
where $c\in \C$ and
$$
F_{(i,l)}(X)= \sum_{j=0}^l H_{(i,j,l-j)}(X).
$$
This completes the proof.
\end{proof}

\begin{proof}[Proof of \thmref{ad-dom}]
As an immediate corollary, from \eqref{mp-der-exp} we see that if 
$A_{[1,i]}-Q_{[1,i]}>0$ for all $1\leq i\leq r$, then the sequence $(u_n)_{n\geq 1}$ in \thmref{Aysm} is convergent.
This completes the proof of \thmref{ad-dom}.
\end{proof}

\begin{ex}\rm
If $(z_1,z_2)=(1,-1)$, then $Q_{[1,1]}=1$ and $Q_{[1,2]}=q_{[1,2]}+q_{[2,2]}=0$. We have
$$
V_2((1,-1))=\{(s_1,s_2)\in \C^2: \Re(s_1)>1,  \Re(s_1+s_2)>0\}.
$$ 
The two variable multiple polylogarithm 
$$
\mathrm{Li}(s_1, s_r;1, -1)=\sum_{n_1>n_2>0}\frac{(-1)^{n_2}}{n_1^{s_1}n_2^{s_2}}
$$
is convergent for $(s_1,s_2)\in V_2((1,-1))\cap \Z^2$. The set $V_2((1,-1))$ is somewhat optimal
for the series $\sum_{n_1>n_2>0}\frac{(-1)^{n_2}}{n_1^{s_1}n_2^{s_2}}$ to converge, as for $N \to \infty$, we have
$$
\sum_{N>n_1>n_2>0}\frac{(-1)^{n_2}}{n_1^{2}n_2^{-2}}=-\frac{\log 2}{2}+
\frac{1}{4}+\frac{(-1)^N}{4}+o(1).
$$
\end{ex}

{\bf Acknowledgements:}
The research of the first author is supported by PMRF (grant number: 1402688, cycle 10).


%
%
%

\newpage

\appendix

\section{Generalised Euler-Boole summation formula}\label{GEBS}

Let $k\geq 1$ be a positive integer and $\zeta$ be a primitive $k^{th}$ root of unity. For $n\in \N$, large enough, the aim
of this appendix is to establish a formula that can be used to estimate a sum of the form $\sum_{a=1}^{n-1}\zeta^af(a)$
for a given `well behaved' function $f$ defined on $[1,n]$, in terms of derivatives and integrals related to $f$.

When $\zeta=1$, we have the well-known \textit{Euler-Maclaurin summation formula} \cite[2.10]{PEP},
\begin{align*}
\sum_{i=1}^{n-1}f(i)=&\int_{1}^{n}f(t)dt+\sum_{j=1}^{m}\frac{B_{j}}{j!}\left(f^{(j-1)}(n)-f^{(j-1)}(1)\right)+\frac{(-1)^{m+1}}{m!}\int_{1}^{n}B_{m}(\{t\})f^{(m)}(t)dt,
\end{align*}
where $B_{j}(x)$ is the $j^{th}$-\textit{Bernoulli polynomial} given by the generating series
$$
\frac{t e^{xt}}{e^t-1}=\sum_{n \ge 0} B_n(x) \frac{t^n}{n!},
$$
with $B_{j}$ denoting the $j^{th}$ \textit{Bernoulli number} given by $B_j=B_j(0)$ and $f(x)$ is an $m$-times continuously differentiable function on $[1,n]$.
 
Similarly, for $\zeta=-1$, one has the \textit{Euler-Boole summation formula} \cite[24.17]{PEP}, 

\begin{align*}
\sum_{i=1}^{n-1}(-1)^if(i)=&\frac{1}{2}\sum_{j=0}^{m-1}\frac{E_j(0)}{j!}\left(-f^{(j)}(1)-(-1)^{n}f^{(j)}(n)\right)
+\frac{1}{2(m-1)!}\int_{1}^{n}f^{(m)}(t)\widetilde E_{m-1}(-t)dt,
\end{align*}
where $E_j(0)$ is the value at $0$ of the \textit{Euler polynomial} $E_j(x)$, given by the generating series
\begin{align*}
\frac{2e^{xt}}{1+e^{t}}=\sum_{n=0}^{\infty} E_{n}(x)\frac{t^n}{n!},
\end{align*}
and $\widetilde E_{n}(x)$ is the \textit{periodic Euler polynomial} defined by $\widetilde E_{n}(x+1)=-\widetilde E_n(x)$ and $\widetilde E_n(x)=E_{n}(x)$ for $0\leq x<1$ (see \cite[24.2]{PEP}).

To prove such a formula for $\sum_{a=1}^{n-1}\zeta^af(a)$ where $\zeta$ is a primitive $k^{th}$
root of unity $\zeta$ for $k\geq 2$, we follow the exposition by Borwein, Calkin and Manna \cite{BCM}. 
We first extend the notion of Euler polynomials. For an integer $k\geq 2$,
we first consider the generating series
\begin{align}
\frac{ke^{xt}}{1+e^{t}+\cdots+e^{(k-1)t}}=\sum_{n=0}^{\infty} E_{k,n}(x)\frac{t^n}{n!},
\end{align}
where $E_{k,n}(x)$ are polynomials in $x$. We call these polynomials as the  \textit{generalised Euler polynomials}.
Note that when $k=2$, we get back the Euler polynomials $E_n(x)$ defined above.
Now for a given primitive $k^{th}$ root of unity $\zeta$, define the corresponding \textit{periodic generalised Euler polynomial}
by setting $\widetilde E_{k,n}(x):=E_{k,n}(x)$ for all $0\leq x<1$ and
$\widetilde E_{k,n}(x+m):=\zeta^{-m}\widetilde E_{k,n}(x)$ for all $x\in \R$ and $m\in \Z$.
This is analogous to the definition of periodic Euler polynomial, when $k=2$ and $\zeta=-1.$ 

To state our theorem, we need some notations. For $k\geq 2$ an integer and integers
$1\leq i\leq j\leq k-1$, consider the vectors 
\begin{align*}
\mathbf{v}_{i,j}:=\left(\zeta^j,\zeta^{j-1},\ldots,\zeta^i\right) \text{ and }
\mathbf{w}_{i,j}:=\left(\frac{i}{k}-1,\frac{i+1}{k}-1,\ldots,\frac{j}{k}-1\right)
\end{align*}
in the inner product space $\C^{j-i+1}$ with the standard inner product $\langle \cdot , \cdot \rangle$. So
$$
\langle \mathbf{v}_{i,j},\mathbf{w}_{i,j}\rangle
=\zeta^j\left(\frac{i}{k}-1\right)+\zeta^{j-1}\left(\frac{i+1}{k}-1\right)+\cdots+\zeta^i\left(\frac{j}{k}-1\right).
$$

\begin{mthm}\label{Mthm}
Let $k\geq 2$ be an integer and $\zeta$ be a primitive $k^{th}$ root of unity. For integers $m \ge 1, n\ge k$, let $f$ be a
complex valued function on $[1,n]$ whose first $m$ many derivatives are absolutely integrable on $[1,n]$.
Then we have the following summation formula:
\begin{align}\label{MFormula}
\begin{split}
\sum_{a=1}^{n-1}\zeta^af(a)&=\frac{1}{k}\left(\sum_{t=1}^{k-1}f(t)\sum_{a=1}^{t}\zeta^a+\zeta^n\sum_{t=0}^{k-2}f(n+t)\sum_{a=t+1}^{k-1}\zeta^{a}\right)\\
&+\sum_{i=1}^{k-2}\langle \mathbf{v}_{1,i},\mathbf{w}_{1,i}\rangle \ \Delta[f](i)
+\zeta^n\sum_{i=2}^{k-1}\langle \mathbf{v}_{i,k-1},\mathbf{w}_{i,k-1}\rangle \ \Delta[f](i+n-2)\\
&+\langle \mathbf{v}_{1,k-1},\mathbf{w}_{1,k-1}\rangle
\sum_{i=1}^{m-1}\frac{E_{k,i}(0)}{i!}\left\{\zeta^{k-1} f^{(i)}(k-1)-\zeta^{n}f^{(i)}(n) \right\}\\
&+\langle \mathbf{v}_{1,k-1},\mathbf{w}_{1,k-1}\rangle
\frac{1}{(m-1)!}\int_{k-1}^{n}\widetilde E_{k,m-1}\left(k-x\right)f^{(m)}(x)dx,
\end{split}
\end{align}
where $\Delta[f](x)=f(x+1)-f(x)$ is the difference operator.
\end{mthm}

Taking $k=2$ and $\zeta=-1$ in \eqref{MFormula}, we get
$\langle \mathbf{v}_{1,1},\mathbf{w}_{1,1}\rangle=\frac{1}{2}$ and hence we get back the
Euler-Boole summation formula, as stated above. Hence, we refer to \eqref{MFormula} as a
\textit{generalised Euler-Boole summation formula}.

\begin{proof}
To prove our theorem, we begin by recalling the general setup from \cite{BCM}. Let $g$ be an absolutely 
continuous probability density function with finite moments. In particular,
$$
\int_{-\infty}^{\infty}g(u)du=1\ \ \  \text{ and }\ \ \ \int_{-\infty}^{\infty}|u|^kg(u)du<\infty \text{ for all } k\in \N.
$$
From now onward, we will denote $\int_{-\infty}^{\infty}$ by $\int$.
Consider the Strodt operator $S_g$ associated with the probability density function $g$ on the
space of complex-valued functions given by
$$
S_g(f)(x):=\int f(x+u)g(u)du.
$$
For an integer $n\geq 0$, let $P_n(x)$ be the corresponding Strodt polynomial for $S_g$, i.e.,
$S_g(P_n)(x)=x^n$ for all $x\in \R$. 
For $1 \le a <n$, a sufficiently smooth function $f$ and a fixed integer $m\geq 0$, define the remainder
$$
R_m(a):=f(a)-\sum_{n=0}^{m}\frac{S_g(f^{(n)})(a)}{k!}P_{n}(0).
$$
As remarked in \cite{BCM}, the process of deriving a summation formula essentially reduces
to finding an expression for $R_m(a)$ as an integral involving the Strodt polynomials $P_n(x)$,
corresponding to the Strodt operator $S_g$.
As $P_0(x)=1$ and $\int g(u)du=1$, we have
\begin{align*}
R_0(a)&=f(a)-\int f(a+u)g(u)du
           =\int (f(a)-f(a+u))g(u)du\\
           &=\int \int_{u}^{0}f'(a+s)g(u)dsdu
           =\int V(s)f'(a+s)ds,
\end{align*}
where
\[ 
V(s):= \left\{
  \begin{array}{lr} 
      \int_{-\infty}^{s}g(u)du & \text{ for } s<0,  \\
       \int_{-\infty}^{s}g(u)du-1 & \text{ for } s\geq 0.
      \end{array}
\right.
\]
Hence  we have
\begin{equation}\label{fa}
f(a)=S_g(f)(a)+\int V(s)f'(a+s)ds.
\end{equation}

For our proof, we consider the Strodt operator $S_g$ associated to the probability weight function
$$
g(u):=\frac{\delta_0(u)+\cdots+\delta_{k-1}(u)}{k}.
$$
So, here the Strodt operator is 
\begin{align}\label{SP}
S_g(f)(x)=\frac{f(x)+f(x+1)+\cdots+f(x+k-1)}{k},
\end{align}
and the generalised Euler polynomials $E_{k,n}(x)$ are the corresponding Strodt polynomials.
Now we have
\begin{align}\label{VS}
\begin{split}
 V(s)&=\left\{
    \begin{array}{ll}
    \frac{1}{k}-1  &  \text{ for } 0\leq s<1,\\
     \frac{2}{k} -1&     \text{ for } 1\leq s<2,\\
     \vdots \\
      \frac{k-1}{k} -1&    \text{ for }  k-2\leq s<k-1,\\
       0          & \text {otherwise}.
    \end{array}
\right.\\
\end{split}
\end{align}
 Then from (\ref{fa}), we get,
\begin{align*}
f(a)=\frac{1}{k}\sum_{i=0}^{k-1}f(a+i)+\int V(s)f'(a+s)ds
 &=\frac{1}{k}\sum_{i=0}^{k-1}f(a+i)+\int_{0}^{k-1} V(s)f'(a+s)ds\\
&= \frac{1}{k}\sum_{i=0}^{k-1}f(a+i)+\int_{a}^{a+k-1} V(x-a)f'(x)dx.
\end{align*} 
We multiply the above expression by $\zeta^a$ and sum for $1 \le a \le n-1$ to get
\begin{equation}\label{a+1}
\sum_{a=1}^{n-1}\zeta^af(a)=\frac{1}{k}\sum_{a=1}^{n-1}\sum_{i=0}^{k-1}\zeta^af(a+i)+
\sum_{a=1}^{n-1}\zeta^a\int_{a}^{a+k-1} V(x-a)f'(x)dx.
\end{equation} 
We begin with the sum,
\begin{align*}
\sum_{a=1}^{n-1}\sum_{i=0}^{k-1}\zeta^af(a+i)&=\sum_{t=1}^{n+k-2}f(t)\sum_{\substack{a+i=t, \\1\leq a\leq n-1,\\ 0\leq i\leq k-1}}\zeta^a\\
&=\sum_{t=1}^{k-1}f(t)\sum_{a=1}^{t}\zeta^a+\sum_{t=k}^{n-1}f(t)\sum_{a=t-k+1}^{t}\zeta^a+\sum_{t=n}^{n+k-2}f(t)\sum_{a=t-k+1}^{n-1}\zeta^a.
\end{align*}
Note that, $\sum_{a=t-k+1}^{t}\zeta^a=\zeta^{t-k+1}(1+\zeta+\cdots+\zeta^{k-1})=0$,
as $\zeta$ is a $k^{th}$ root of unity. Hence the second sum in the right-hand side of the above equation vanishes. Moreover,
in the last sum, replacing $t$ by $n+t$, we have
 \begin{align}\label{subsum}
\sum_{a=1}^{n-1}\sum_{i=0}^{k-1}\zeta^af(a+i)
&=\sum_{t=1}^{k-1}f(t)\sum_{a=1}^{t}\zeta^a+\sum_{t=0}^{k-2}f(n+t)\sum_{a=t+n-k+1}^{n-1}\zeta^a \nonumber \\
&=\sum_{t=1}^{k-1}f(t)\sum_{a=1}^{t}\zeta^a+\zeta^n\sum_{t=0}^{k-2}f(n+t)\sum_{a=t+1}^{k-1}\zeta^{a}.
\end{align}

Now we consider the last term on the right-hand side of \eqref{a+1}. For $0 \le i \le k-2$, if $x \in [a+i,a+i+1)$, we have
$V(x-a)= \frac{i+1}{k}-1$. So,
\begin{align}\label{VS1}
\begin{split}
&\sum_{a=1}^{n-1}\zeta^a \int_{a}^{a+k-1} V(x-a)f'(x)dx\\
&=\sum_{a=1}^{n-1}\zeta^a\left(\int_{a}^{a+1}\left(\frac{1}{k}-1\right)f'(x)dx+\cdots+
\int_{a+k-2}^{a+k-1}\left(\frac{k-1}{k}-1\right)f'(x)dx\right).
\end{split}
\end{align}
Note that in \eqref{VS1}, we have $\Delta[f](i)=\int_{i}^{i+1}f'(x)dx$ for $1\leq i\leq n+k-3$.
We will compute the coefficient of each of these integrals. 
We make three cases: (a) $1\leq i\leq k-2$, (b) $k-1\leq i\leq n-1$ and (c) $n\leq i\leq n+k-3$.\\
\textbf{Case (a).} In this case, the coefficient of $\Delta[f](i)$ in \eqref{VS1} is
$$
\zeta^i\left(\frac{1}{k}-1\right)+\zeta^{i-1}\left(\frac{2}{k}-1\right)+\cdots+\zeta\left(\frac{i}{k}-1\right)
=\langle \mathbf{v}_{1,i},\mathbf{w}_{1,i}\rangle.
$$
\textbf{Case (b).} In this case, the coefficient of the integral $\int_{i}^{i+1}f'(x)dx$ in \eqref{VS1} is
\begin{align*}
&\zeta^i\left(\frac{1}{k}-1\right)+\zeta^{i-1}\left(\frac{2}{k}-1\right)+\cdots+\zeta^{i-(k-2)}\left(\frac{k-1}{k}-1\right)\\
=&\zeta^{i-(k-1)}\left\{\zeta^{k-1}\left(\frac{1}{k}-1\right)+\zeta^{k-2}\left(\frac{2}{k}-1\right)+\cdots+\zeta\left(\frac{k-1}{k}-1\right)\right\}\\
=&\zeta^{i+1}\langle \mathbf{v}_{1,k-1},\mathbf{w}_{1,k-1}\rangle.
\end{align*}
\textbf{Case (c).} In this case,  the coefficient of $\Delta[f](i)$ in \eqref{VS1} is
\begin{align*}
&\zeta^{n-1}\left(\frac{i-n+2}{k}-1\right)+\zeta^{n-2}\left(\frac{i-n+3}{k}-1\right)+\cdots+\zeta^{i-k+2}\left(\frac{k-1}{k}-1\right)\\
=&\zeta^{n}\left\{\zeta^{k-1}\left(\frac{i-n+2}{k}-1\right)+\zeta^{k-2}\left(\frac{i-n+3}{k}-1\right)+\cdots+\zeta^{i-n+2}\left(\frac{k-1}{k}-1\right)\right\}\\
=&\zeta^n\langle \mathbf{v}_{i-n+2,k-1},\mathbf{w}_{i-n+2,k-1}\rangle.
\end{align*}
Hence \eqref{VS1} becomes
\begin{align}\label{VS2}
\begin{split}
&\sum_{a=1}^{n-1}\zeta^a \int_{a}^{a+k-1} V(x-a)f'(x)dx\\
&=\sum_{i=1}^{k-2}\langle \mathbf{v}_{1,i},\mathbf{w}_{1,i}\rangle \ \Delta[f](i)
+\langle \mathbf{v}_{1,k-1},\mathbf{w}_{1,k-1}\rangle\sum_{i=k-1}^{n-1}\int_{i}^{i+1}\zeta^{i+1}f'(x)dx\\
& \ \ + \zeta^n\sum_{i=n}^{n+k-3}\langle \mathbf{v}_{i-n+2,k-1},\mathbf{w}_{i-n+2,k-1}\rangle \ \Delta[f](i).
\end{split}
\end{align}
Clearly,
\begin{align}\label{Midstep}
\begin{split}
\zeta^n\sum_{i=n}^{n+k-3}\langle \mathbf{v}_{i-n+2,k-1},\mathbf{w}_{i-n+2,k-1}\rangle \ \Delta[f](i)
=\zeta^n\sum_{i=2}^{k-1}\langle \mathbf{v}_{i,k-1},\mathbf{w}_{i,k-1}\rangle \ \ \Delta[f](i+n-2).
\end{split}
\end{align}
Now we will deal with the term $\sum_{i=k-1}^{n-1}\int_{i}^{i+1}\zeta^{i+1}f'(x)dx$ in \eqref{VS2},
using the general theory of finite Strodt operator and Strodt polynomials as in \cite{BCM}.
We compute the periodic generalised Euler polynomial $\widetilde E_{k,0}(x)$. For $i\in \Z$, we have
\begin{align*}
\widetilde E_{k,0}\left(k-x\right)=\widetilde E_{k,0}\left(k-(1+i)+1+i-x\right)
                                                    =\zeta^{1+i-k}\widetilde E_{k,0}\left(1+i-x\right)
                                                    =\zeta^{i+1}\widetilde E_{k,0}\left(1+i-x\right).
\end{align*}
If $x \in (i,i+1)$, we have $1+i-x \in (0,1)$. Therefore $\widetilde E_{k,0}\left(1+i-x\right)=E_{k,0}\left(1+i-x\right)=1$
for $x \in (i,i+1)$. Hence for $x \in (i,i+1)$,
\begin{align*}
\widetilde E_{k,0}\left(k-x\right)=\zeta^{i+1}.
\end{align*}
Hence we have
\begin{align}\label{3T}
\begin{split}
\sum_{i=k-1}^{n-1}\int_{i}^{i+1}\zeta^{i+1}f'(x)dx&=\sum_{i=k-1}^{n-1}\int_{i}^{i+1}\widetilde E_{k,0}\left(k-x\right)f'(x)dx
 =\int_{k-1}^{n}\widetilde E_{k,0}\left(k-x\right)f'(x)dx.
 \end{split}
\end{align}
Now we recall \cite[Theorem 2.3]{BCM} that the Strodt polynomials satisfy the differential equation
\begin{align*}
\frac{d}{dx}E_{k,n}(x)=nE_{k,n-1}(x),
\end{align*}
for $n \ge 1$. So this implies that,
\begin{align}\label{DE}
\frac{d}{dx}E_{k,n}(k-x)=-nE_{k,n-1}(k-x).
\end{align}
Hence for an integer $m\geq 1$, using \eqref{DE} for the periodic generalised Euler polynomials
and integrating by parts, we get
\begin{align}\label{IBP}
\begin{split}
\int_{k-1}^{n}\widetilde E_{k,m-1}\left(k-x\right)f^{(m)}(x)dx=&\frac{1}{m}\left\{\widetilde E_{k,m}\left(1\right)f^{(m)}(k-1)-\widetilde E_{k,m}\left(k-n\right)f^{(m)}(n)\right\}\\
&+\frac{1}{m}\int_{k-1}^{n}\widetilde E_{k,m}\left(k-x\right)f^{(m+1)}(x)dx.
\end{split}
\end{align}
Since $\widetilde E_{k,m}\left(k-n\right)=\zeta^{n-k}E_{k,m}\left(0\right)=\zeta^{n}E_{k,m}\left(0\right)$ and $\widetilde E_{k,m}\left(1\right)=\zeta^{-1} E_{k,m}\left(0\right)$, \eqref{IBP} becomes
\begin{align}\label{IBP2}
\begin{split}
\int_{k-1}^{n}\widetilde E_{k,m-1}\left(k-x\right)f^{(m)}(x)dx=
&\frac{E_{k,m}(0)}{m}\left\{\zeta^{k-1} f^{(m)}(k-1)-\zeta^{n}f^{(m)}(n)\right\}\\
&+\frac{1}{m}\int_{k-1}^{n}\widetilde E_{k,m}\left(k-x\right)f^{(m+1)}(x)dx.
\end{split}
\end{align}
Applying \eqref{IBP2} repeatedly, \eqref{3T} becomes
\begin{align}\label{4T}
\begin{split}
\sum_{i=k-1}^{n-1}\int_{i}^{i+1}\zeta^{i+1}f'(x)dx&=
\sum_{i=1}^{m-1}\frac{E_{k,i}(0)}{i!}\left\{\zeta^{k-1} f^{(i)}(k-1)-\zeta^{n}f^{(i)}(n) \right\}\\
&+\frac{1}{(m-1)!}\int_{k-1}^{n}\widetilde E_{k,m-1}\left(k-x\right)f^{(m)}(x)dx.
 \end{split}
\end{align}
Now in \eqref{a+1}, we apply \eqref{subsum} and \eqref{VS2}, and use \eqref{Midstep} and \eqref{4T}, to get the desired 
formula \eqref{MFormula}. This completes the proof of the theorem.
\end{proof}

\end{document}